\documentclass[12pt]{amsart}
\usepackage{amsmath}
\usepackage[backref=page]{hyperref}
\usepackage{tikz,float,caption}
\usetikzlibrary{arrows.meta}

\usepackage[margin=1.25 in,top=1.2in, bottom=1.4 in]{geometry}

\usepackage[bb=ams,scr=euler]{mathalpha}

\usepackage{setspace}
\usepackage{enumitem}
\setenumerate{
  wide=10pt,
  leftmargin=0pt,
  labelwidth=*,
  label=(\roman*),
  topsep=0pt,
  listparindent=0pt,
  parsep=4pt}

\usepackage{thmtools}
\declaretheoremstyle[headfont=\normalsize\normalfont\bfseries,notefont=\mdseries,
notebraces={(}{)},bodyfont=\normalfont,postheadspace=0.5em]{basicstyle}
\declaretheoremstyle[headfont=\normalsize\normalfont\bfseries,notefont=\mdseries,
notebraces={(}{)},bodyfont=\normalfont\itshape,postheadspace=0.5em]{italstyle}

\declaretheorem[name=Definition,style=basicstyle]{defn}

\declaretheorem[style=italstyle,name=Theorem]{theorem}
\declaretheorem[style=italstyle,name=Question]{question}

\declaretheorem[style=italstyle,name=Lemma]{lemma}

\declaretheorem[style=italstyle,name=Proposition,sibling=lemma]{prop}

\declaretheorem[style=basicstyle,name=Proof,numbered=no,qed=$\square$]{preproof}
\renewenvironment{proof}{\preproof}{\endpreproof}

\makeatletter
\newcommand{\C}{\mathbb{C}}
\renewcommand{\d}{\mathrm{d}}

\newcommand{\pr}{\mathrm{pr}}
\newcommand{\R}{\mathbb{R}}
\newcommand{\set}[1]{\left\{#1\right\}}
\renewcommand{\subsection}{\@startsection{subsection}{2}%
  \z@{.5\linespacing\@plus.7\linespacing}{-.5em}%
  {\normalfont\itshape}}

\newcommand{\Z}{\mathbb{Z}}

\makeatother

\title{Contact non-squeezing and orderability via the shape invariant}
\author{Dylan Cant}
\date{\today}
\begin{document}
\maketitle

\begin{abstract}
  We prove a contact non-squeezing result for a class of embeddings between starshaped domains in the contactization of the symplectization of the unit cotangent bundle of certain manifolds. The class of embeddings includes embeddings which are not isotopic to the identity. This yields a new proof that there is no positive loop of contactomorphisms in the unit cotangent bundles under consideration. The proof uses the shape invariant introduced by Sikorav and Eliashberg.
\end{abstract}

\section{Introduction}
\label{sec:introduction}

Let $(Y,\xi)$ be a compact cooriented contact manifold, let $(SY,\lambda)$ denote its symplectization, and let $Q=\R/\Z\times SY$ be the contactization with contact form $\alpha=\d t-\lambda$. The philosophy propounded by \cite{ekp} is that \emph{orderability} properties of $Y$ are related to \emph{contact non-squeezing} properties of certain domains in $Q$.

The relevant class of domains are the \emph{starshaped} ones, namely those defined as sublevel sets $E\le 1$ where $E:Q\to (0,\infty)$ is homogeneous with repect to the Liouville flow, i.e., $E(t,e^{s}z)=e^{s}E(t,z)$. Write $e^{s}z$ for the image of $z\in SY$ under the time $s$ Liouville flow.

Recall from \cite{ep2000} that a path of contactomorphisms on $Y$ is \emph{positive} provided its lift to a path of equivariant symplectomorphisms of $SY$ is generated by a positive homogeneous Hamiltonian function.

We say that $(Y,\xi)$ is \emph{orderable} if it does not admit a positive contractible loop of contactomorphisms, while we say it is \emph{strongly orderable}\footnote{In \cite{casals_presas_strong} the authors introduce the term strongly orderable. However, the term ``strongly orderable'' is also defined in \cite{guogang_liu_positive_loops} and \cite{chantraine_colin_rizell} and means something else. It does not seem like there is a standard term for contact manifolds which admit a positive (potentially non-contractible) loop of contactomorphisms.} if it does not admit \emph{any} positive loop of contactomorphisms.


One of the results of \cite{ekp} can be paraphrased as follows: {\itshape (i) if $Y$ is not strongly orderable, then any starshaped domain can be contactomorphically embedded into any other sharshaped domain; (ii) if $Y$ is not orderable, then any sufficiently small starshaped domain can be transported by a compactly supported contact isotopy into any other starshaped domain.} Morally, positive loops of contactomorphisms can be used a squeezing tool.

Moreover, the arguments of \cite{ekp} imply the existence of contact embeddings satisfying a particular cohomological property. Observe that the projection of any starshaped domain $\Omega$ onto $\R/\Z\times Y$ induces a canonical isomorphism between $H^{1}_{\mathrm{dR}}(\Omega)$ and $H^{1}_{\mathrm{dR}}(Q)$. This allows us to define the following class of embeddings of starshaped domains:
\begin{defn}
  Let $\Omega_{1},\Omega_{2}$ be starshaped domains. A contact embedding $\Phi:\Omega_{1}\to \Omega_{2}$ is called a \emph{generalized squeezing} provided there is a function $k:H^{1}_{\mathrm{dR}}(Q)\to \R$ so that:
  \begin{equation}\label{eq:gen_squeeze}
    \Phi^{*}[\beta]=[\beta]+k([\beta])[\d t],
  \end{equation}
  and $k([\d t])=0$, where $t:Q\to \R/\Z$ is the coordinate projection.
\end{defn}
Part (i) of the \cite{ekp} construction described above produces generalized squeezings of starshaped domains:
\begin{lemma}\label{lemma:pos_loop_squeezing}
  If $Y$ is connected and not strongly orderable, then any starshaped domain admits a generalized squeezing into any other starshaped domain.
\end{lemma}

The main purpose of this paper is a non-squeezing result. Let $Y$ be the unit cotangent bundle of the $n$-torus, $n>1$, and let $Q=\R/\Z\times SY$. Identify $SY$ with the complement of the zero section in $\mathbb{T}^{n}\times \R^{n}$, and let $q_{1},\dots,q_{n},p_{1},\dots,p_{n}$ be the resulting coordinates. Observe that $\Omega(r)=\set{0<p_{1}^{2}+\dots+p_{n}^{2}\le r^{2}}$ defines a particular starshaped domain in $Q$. We will prove:
\begin{theorem}\label{theorem:nonsqueezing}
  There is no generalized squeezing of $\Omega(R)$ into $\Omega(r)$ provided $r<R$.
\end{theorem}
The tool used to prove Theorem \ref{theorem:nonsqueezing} is the \emph{shape invariant} for subsets of exact symplectic manifolds, as in \cite{sikorav_shape_duke}, \cite{e_new_invariants}, \cite{sikorav_shape_1}, \cite{ep2000}, \cite{muller_spaeth}.

On the other hand, if we consider all contact embeddings, there is total flexibility:
\begin{theorem}\label{theorem:flex}
  Let $Y$ be the unit cotangent bundle of the $n$-torus, let $\Omega\subset \R/\Z\times SY$ be any starshaped domain, and let $M^{2n+1}$ be any other contact manifold (e.g., take $M$ to be a Darboux chart). Then there is a contact embedding $\varphi:\Omega\to M$.
\end{theorem}
This follows from \cite[Corollary 1.25]{ekp} whose proof implies that $\R/\Z\times \C^{n}$, with the contact form $\d t - \frac{1}{2}\sum x_{i}\d y_{i}-y_{i}\d x_{i}$, admits a contact embedding into a Darboux chart. 

Theorem \ref{theorem:nonsqueezing} can be generalized. Let $B^{n}$, $n>1$, be a compact and connected smooth manifold, let $Y$ be the unit cotangent bundle of $B$, and let $Q=\R/\Z\times SY$. Considering $SY\subset T^{*}B$ as the subset of the zero section, a fiberwise metric $\rho:T^{*}B\to [0,\infty)$ determines a star-shaped domain $\Omega(r)=\set{(t,z)\in Q:0<\rho(z)\le r^{2}}$. Then:
\begin{theorem}\label{theorem:nonsqueezing_gen}
  If $B$ admits a closed one-form $\beta$ such that $\rho(\beta)=1$ holds at all points then there is no generalized squeezing of $\Omega(R)$ into $\Omega(r)$ if $r<R$.
\end{theorem}

As a corollary to Theorem \ref{theorem:nonsqueezing_gen} and Lemma \ref{lemma:pos_loop_squeezing}, we conclude the known result that the unit cotangent bundle of closed manifolds $B$ with nowhere zero one-forms are strongly orderable (for $n>1$); see \cite[Theorem 1.3.B]{ep2000} (for orderability), and \cite{colin_ferrand_pushkar}, \cite{chernov_nemirovski_nonneg}, \cite{chernov_nemirovski_universal} (for strong orderability). Of course, the unit cotangent bundle of $\mathbb{T}^{1}$ is not strongly orderable (although it is orderable).

The natural question suggested by Theorem \ref{theorem:nonsqueezing} and \ref{theorem:nonsqueezing_gen} is:
\begin{question}
  Given a strongly orderable contact manifold $Y$, do star-shaped domains in $Q$ satisfy a non-squeezing result for generalized squeezings?
\end{question}

\subsection{Which contact manifolds are known to be strongly orderable?}
The unit cotangent bundle of any manifold $M^{d}$, $d>1$, with an open universal cover is strongly orderable. This is established in \cite{colin_ferrand_pushkar,chernov_nemirovski_nonneg,chernov_nemirovski_universal} where they show that there is no positive loop of Legendrians based at a fiber in $ST^{*}M$. The condition that $M$ has an open universal cover is weaker than $M$ admitting a nowhere zero closed one-form, e.g., a surface of genus $g>1$ does not admit a nowhere zero closed one-form. In \cite{chernov_nemirovski_nonneg}, the authors use \cite[Theorem 1.18]{ekp} and prove that all unit cotangent bundles are orderable. 

By measuring the growth rate of Rabinowitz Floer homology groups associated to a positive loop of contactomorphisms, \cite{albers_frauenfelder_nonlinear_maslov} shows that any closed manifold $M$ with finite fundamental group and a rational cohomology ring with at least two generators has a strongly orderable unit cotangent bundle. Note that such manifolds do not have an open universal cover.

The work of \cite{weigel} describes a local modification one can do to any Liouville fillable contact manifold $Y^{2n+1}$, $n\ge 3$, which produces a strongly orderable contact manifold.

On the other hand, a prequantization space is never strongly orderable as the Reeb flow induces a positive loop, although in many cases they are known to be orderable, see \cite{ep2000,ekp,milin_dissertation,sandon_orderability_lens,albers_merry_orderability_non_squeezing}.

In another direction, the work of \cite{casals_presas_strong} shows that if $(Y^{3},\xi)$ admits a positive loop of contactomorphisms $\varphi_{t}$, then (under certain circumstances) there is a transverse knot $\kappa$ so that the half Lutz twist along $\kappa$ produces an overtwisted contact manifold which is not strongly orderable. This is a step towards understanding the open question of the relationship between orderability and overtwistedness; see \cite[\S10]{BEM}, \cite{casals_presas_sandon_small}, and \cite{guogang_liu_positive_loops} for related results.

\subsection*{Acknowledgements}
I want to thank Jakob Hedicke for introducing me to the concept of orderability versus strong orderability, and for many discussions surrounding these topics, and Eric Kilgore for discussions on the contact shape invariant. This work was completed at the University of Montreal with funding from the CIRGET research group.

\section{Proofs}
\label{sec:proofs}
The outline for the rest of the paper is as follows: In \S\ref{sec:ekp_construction} positive loops are shown to induce generalized squeezings (Lemma \ref{lemma:pos_loop_squeezing}), in \S\ref{sec:shape_invariant} the shape invariant is reviewed, and in \S\ref{sec:proof_of_theorem}, \S\ref{sec:proof_of_flex}, \S\ref{sec:proof_of_gen}, Theorem \ref{theorem:nonsqueezing}, Theorem \ref{theorem:flex}, and Theorem \ref{theorem:nonsqueezing_gen} are proved, respectively.

\subsection{Positive loops as a squeezing tool}
\label{sec:ekp_construction}
Suppose that $Y$ is non-orderable. Let $\varphi_{t}$ be a positive loop of contactomorphisms, considered as a positive loop of equivariant symplectomorphisms of $(SY,\lambda)$. Without loss of generality, suppose that $\varphi_{0}=1$.

Positivity means that $\dot\varphi_{t}=X_{t}\circ \varphi_{t}$ where $X_{t}$ is the Hamiltonian vector field for a positive homogeneous function $H(t,z)=H_{t}(z)$ on $\R/\Z\times SY$. 

\begin{lemma}[Proposition 2.1 in \cite{ekp}]\label{lemma:ekp_prop}
  For $h_{t}(z)=-\log(1+H_{t}(\varphi_{t}(z)))<0$, the map $\Phi:(t,z)\mapsto (t,e^{h_{t}(z)}\varphi_{t}(z))$ is a contact embedding of $Q=\R/\Z\times SY$ into the interior of starshaped domain $\set{(t,z):H_{t}(z)\le 1}$.\hfill$\square$
\end{lemma}

We will use this to prove Lemma \ref{lemma:pos_loop_squeezing}.
\begin{proof}[of Lemma \ref{lemma:pos_loop_squeezing}]
  Let $\varphi_{t}$ be a positive loop of contactomorphisms. Fix two starshaped domains $\Omega_{1},\Omega_{2}$ in $Q$. For $k\in \Z_{>0}$, $\varphi_{kt}$ is a positive loop of contactomorphisms, whose Hamiltonian equals\footnote{Thanks to Jakob Hedicke for pointing out this trick to obtain positive loops with arbitrarily large contact Hamiltonians.} $kH_{kt}$. In particular, by taking $k$ sufficiently large, we may suppose that $kH_{kt}(z)<1$ implies $(t,z)\in \Omega_{2}$. Then the map $\Phi$ defined in Lemma \ref{lemma:ekp_prop} using the positive loop $\varphi_{kt}$ maps $\Omega_{1}$ into $\Omega_{2}$. 
  
  It remains to check the condition on the cohomology classes \eqref{eq:gen_squeeze}. It is sufficent to prove that $\Psi(t,z)=(t,\varphi_{kt}(z))$ satisfies \eqref{eq:gen_squeeze}, since $\Psi$ differs from $\Phi$ by a smooth homotopy. By assumption $\varphi_{0}=1$, and hence any loop $\gamma$ contained in the slice $\set{0}\times SY$ satisfies $\Psi\circ \gamma=\gamma$. For such loops we have $\int \gamma^{*}\beta=\int \gamma^{*}\Psi^{*}\beta$ for any one-form $\beta$.

  Let $\eta(t)=(t,z)$, for any $z\in SY$, and let $k(\beta)=\int \eta^{*}\Psi^{*}\beta-\eta^{*}\beta$. This does not depend on $z$ since $Y$ is connected (incidentally, the unit cotangent bundle of $\mathbb{T}^{1}$ is not connected, and this is where the proof fails for $n=1$). For all $\beta,\gamma,\eta$ as above:
  \begin{equation*}
    \int \eta^{*}\Psi^{*}\beta=\int \eta^{*}(\beta+k(\beta)\d t)\text{ and }\int \gamma^{*}\Psi^{*}\beta=\int \gamma^{*}(\beta+k(\beta)\d t).
  \end{equation*}
  A standard argument shows that $\Psi^{*}[\beta]=[\beta]+k(\beta)[\d t]$ in $H^{1}_{\mathrm{dR}}$ (roughly speaking, decompose any homotopy class of loops into a concatenation of pieces of the form $\eta$ and $\gamma$, and then use the fact that exact $1$-forms are characterized by having integral zero over every loop).

  On the other hand, it is immediate that $\Psi^{*}[\d t]=[\d t]$, and hence $\beta([\d t])=0$. This completes the proof of Lemma \ref{lemma:pos_loop_squeezing}.
\end{proof}

\subsection{Review of the shape invariant}
\label{sec:shape_invariant}

Consider $T^{*}\mathbb{T}^{n+1}\simeq \mathbb{T}^{n+1}\times \R^{n+1}$ with $\R/\Z$-valued spatial coordinates $x_{0},\dots,x_{n}$ and momentum coordinates $y_{0},\dots,y_{n}$. Let $\lambda=\sum y_{i}\d x_{i}$ be the canonical one-form. Say that a Lagrangian embedding $e:\mathbb{T}^{n+1}\to T^{*}\mathbb{T}^{n+1}$ is \emph{homologically standard} provided $e^{*}\pr^{*}=\mathrm{id}$ on $H^{1}_{\mathrm{dR}}(\mathbb{T}^{n+1})$, where $\pr$ is the projection.

Let $U\subset T^{*}\mathbb{T}^{n+1}$. Following \cite{sikorav_shape_duke}, \cite{e_new_invariants}, \cite{sikorav_shape_1}, \cite{ep2000}, define:
\begin{equation*}
  \mathrm{Shape}(U):=\set{[e^{*}\lambda]:e\text{ is homologically standard and $e(\mathbb{T}^{n+1})\subset U$}}\subset H^{1}_{\mathrm{dR}}(\mathbb{T}^{n+1}).
\end{equation*}
This uses the fact that $e^{*}\lambda$ is a closed one-form when $e$ is Lagrangian.

Given two subsets $U_{1},U_{2}$ of $T^{*}\mathbb{T}^{n+1}$, a map $\varphi:U_{1}\to U_{2}$ is \emph{homologically standard} provided $\varphi^{*}\mathrm{pr}^{*}=\mathrm{pr}^{*}$ on $H^{1}_{\mathrm{dR}}(\mathbb{T}^{n+1})$. We have the following result for \emph{exact} symplectic embeddings, i.e., those satisfying $\varphi^{*}\lambda=\lambda+\d f$.
\begin{lemma}\label{lemma:homologically_standard_exact}
  If $\varphi:U_{1}\to U_{2}$ is a homologically standard and exact symplectic embedding, then $\mathrm{Shape}(U_{1})\subset \mathrm{Shape}(U_{2})$.
\end{lemma}
\begin{proof}
  Let $e:\mathbb{T}^{n+1}\to U_{1}$ be homologically standard. Then $\varphi\circ e$ is homologically standard in $U_{2}$, and hence $[e^{*}\lambda]$ is in the shape of $U_{2}$. Since $e$ was arbitrary, we conclude the desired result.
\end{proof}

The constant $1$-forms $\sum a_{i}\d x_{i}$ form a canonical basis $H^{1}_{\mathrm{dR}}(\mathbb{T}^{n+1})\simeq \R^{n+1}$. The fundamental theorem about shapes is the following result, originally due to \cite{sikorav_shape_duke}:

\begin{prop}\label{prop:sikorav_shape}
  Let $U=\mathbb{T}^{n+1}\times W$ where $W\subset \R^{n+1}$. Then, using the standard basis for $H^{1}_{\mathrm{dR}}(\mathbb{T}^{n+1})$, the shape of $U$ is equal to $W$.
\end{prop}
The proof is a straightforward application of the fact that every closed exact Lagrangian in the cotangent bundle intersects the zero section. See \cite[Chapter 9]{mcduffsalamon} and the references therein for more details. 

\subsection{Proof of Theorem \ref{theorem:nonsqueezing}}
\label{sec:proof_of_theorem}
Throughout this section, let $Y=\mathbb{T}^{n}\times S^{n-1}$, with coordinates $q_{1},\dots,q_{n}$ and $p_{1},\dots,p_{n}$, with $p_{1}^{2}+\dots+p_{n}^{2}=1$, and with the unit cotangent bundle contact form $\alpha=\sum p_{i}\d q_{i}$. Identify $SY$ with the open subset of $\mathbb{T}^{n}\times \R^{n}$ obtained by removing the zero section $\set{p=0}$. The Liouville flow by time $s$ is given by $(q,p)\mapsto (q,e^{s}p)$. Therefore:
\begin{equation*}
  Q=\R/\Z\times \mathbb{T}^{n}\times \R^{n}\text{ with form }\d t-\sum p_{i}\d q_{i}.
\end{equation*}
Recall the starshaped domain $\Omega(r):=\set{(t,q,p):0<\sum p_{i}^{2}\le r^{2}}$. Our goal is to prove that there is no generalized squeezing $\Omega(R)\to \Omega(r)$ if $R>r$.

In order to apply the shape invariant, we will embed $Q$ into $T^{*}\mathbb{T}^{n+1}$ as a contact type hypersurface. Define an embedding by:
\begin{equation*}
  \begin{aligned}
    x_{0}=t\text{ and }x_{i}&=q_{i}\\
    y_{0}=1\text{ and }y_{i}&=-p_{i}
  \end{aligned}\hspace{.5cm}\text{ for }i=1,\dots,n.
\end{equation*}
The restriction of the canonical form to this hypersurface equals $\d t-\sum p_{i}\d q_{i}$.

Let $U(r)=\set{(x,y):0<\sum y_{i}^{2}\le r^{2}y_{0}^{2}\text{ and }y_{0}>0}$, so that $U(r)$ is the cone over $\Omega(r)$. By standard properties of contact type hypersurfaces, if there is a contact embedding $\Phi:\Omega(R)\to \Omega(r)$, then there is an equivariant symplectomorphism $\varphi:U(R)\to U(r)$, i.e., $\varphi(x,e^{s}y)=e^{s}\varphi(x,y)$, lifting $\Phi$. It is well-known that equivariant symplectomorphisms are exact.

If $\Phi$ is a generalized squeezing, then $\varphi$ will not generally be homologically standard, and so we cannot yet apply Lemma \ref{lemma:homologically_standard_exact} to $U(R)$ and $U(r)$. The idea is to ``correct'' $\varphi$ to make it homologically standard by postcomposing with a canonical transformation of $T^{*}\mathbb{T}^{n+1}$.

Let $\Phi^{*}([\beta])=[\beta]+k([\beta])\d t$ and let $k_{i}=k([\d q_{i}])$. It is clear that:
\begin{equation*}
  \varphi^{*}[\d x_{0}]=[\d x_{0}]\text{ and }\varphi^{*}[\d x_{i}]=[\d x_{i}]+k_{i}[\d x_{0}].
\end{equation*}
Consider the canonical transformation:
\begin{equation*}
  \delta(x,y)=(x_{0},x_{1}-k_{1}x_{0},\dots,x_{n}-k_{1}x_{0},y_{0}+\sum k_{i}y_{i},y_{1},\dots,y_{n}).
\end{equation*}
Observe that $\delta^{*}(\sum y_{j}\d x_{j})=\sum y_{j}\d x_{j}$ and $\varphi^{*}\delta^{*}[\d x_{j}]=[\d x_{j}]$ for $j=0,\dots,n.$ Therefore $\delta\circ \varphi$ \emph{is a homologically standard exact embedding of $U(R)$ into $\delta(U(r))$}, and Lemma \ref{lemma:homologically_standard_exact} implies the shape of $U(R)$ is contained in the shape of $\delta(U(r))$.

It is clear that $U(R)=\mathbb{T}^{n+1}\times W(R)$ and $\delta(U(r))=\mathbb{T}^{n+1}\times W_{k}(r)$ with:
\begin{equation*}
  \begin{aligned}
    W(R)&=\set{0< \textstyle\sum_{i=1}^{n} y_{i}^{2}\le R^{2}y_{0}^{2}\text{ and }y_{0}>0}\\
    W_{k}(r)&=\set{0<\textstyle\sum y_{i}^{2}\le r^{2}(y_{0}-\sum k_{i}y_{i})^{2}\text{ and }y_{0}-\sum k_{i}y_{i}>0}.
  \end{aligned}
\end{equation*}
Proposition \ref{prop:sikorav_shape} implies the shapes of $U(R)$ and $\delta(U(r))$ are $W(R)$ and $W_{k}(r)$, respectively. It is easy to see that, no matter what $k_{1},\dots,k_{n}$ are, that $W(R)$ is not contained\footnote{See the proof of Theorem \ref{theorem:nonsqueezing_gen} below for a related argument.} in $W_{k}(r)$, and so applying Lemma \ref{lemma:homologically_standard_exact} shows $\varphi$, and hence $\Phi$, cannot exist. This concludes the proof of Theorem \ref{theorem:nonsqueezing}.

\subsection{Proof of Theorem \ref{theorem:flex}}
\label{sec:proof_of_flex}
Let $Y$ be the unit cotangent bundle of $\mathbb{T}^{n}$ and suppose that $\Omega\subset \R/\Z\times SY$ is a star-shaped domain. Consider $SY\subset T^{*}\mathbb{T}^{n}$ as the complement of the zero section. Without loss of generality, suppose $\Omega\subset \set{\min(p_{1},\dots,p_{n})>-k}$ for some integer $k>0$. As mentioned in \S\ref{sec:introduction}, it suffices to construct a contact embedding $\Omega\to \R/\Z\times \mathbb{C}^{n}$, since $\R/\Z\times \mathbb{C}^{n}$ has a contact embedding into a Darboux chart by \cite[Corollary 1.25]{ekp}.

Using the $t,q,p$ coordinate system, consider the map $\Omega\to \R/\Z\times \C^{n}$ given by:
\begin{equation*}
  \varphi:(t,q_{1},\dots,q_{n},p_{1},\dots,p_{n})\mapsto (t+k\sum q_{i},f(p_{1})e^{2\pi i q_{1}},\dots,f(p_{n})e^{2\pi i q_{n}}),
\end{equation*}
where $f(p)=\pi^{-1/2}(p+k)^{1/2}$ (this is why we require $\min(p_{1},\dots,p_{n})>-k$). Then $\varphi$ is a strict contact embedding. It is easily seen to be injective. We compute:
\begin{equation*}
  \varphi^{*}(\d t-\frac{1}{2}\sum_{i=1}^{n}x_{i}\d y_{i}-y_{i}\d x_{i})=\d t+k\sum_{i=1}^{n} \d q_{i}-\sum_{i=1}^{n}\pi f(p_{i})^{2}\d q_{i}=\d t-\sum_{i=1}^{n} p_{i}\d q_{i},
\end{equation*}
as desired. This completes the proof of Theorem \ref{theorem:flex}.

\subsection{Generalization to some other cotangent bundles}
\label{sec:proof_of_gen}

Let $\rho$ be a fiberwise metric so that the nowhere zero one-form $\beta$ satisfies $\rho(\beta)=1$. Let $\Omega(r)=\set{0<\rho(z)\le r^{2}}$ be considered as starshaped domain in $\R/\Z\times SY$, where $Y$ is the unit cotangent bundle of $B$. In search of a contradiction, suppose there is a generalized squeezing $\Phi:\Omega(R)\to \Omega(r)$ for $r<R$.

Consider the embedding: $$(t,z)\in \R/\Z \times T^{*}B\mapsto (t,1,-z)\in T^{*}(\R/\Z\times B)=T^{*}\R/\Z\times T^{*}B.$$
Writing $s$ for the fiber coordinate on $T^{*}\R/\Z$, so the canonical form is $s\d t$, we conclude that this defines a contact-type hypersurface inside of $T^{*}(\R/\Z\times B)$ whose induced contact form equals $\d t-\lambda$.

Extending via the Liouville flow, the contact embedding $\Phi:\Omega(R)\to \Omega(r)$ lifts to an equivariant exact symplectomorphism $\varphi:U(R)\to U(r)$, where:
\begin{equation*}
  U(r):=\R_{+}\Omega(r)=\set{(t,s,z):0<\rho(z)/s^{2}\le r^{2}\text{ and }s<0},
\end{equation*}
and similarly for $U(R)$. Since $\Phi$ is assumed to be a generalizing squeezing, it is easy to see that $\varphi^{*}[\d t]=[\d t]$ and $\varphi^{*}[\beta]=[\beta]+k[\d t]$.

Consider the Lagrangian graph $L_{a,b}=a\d t+b\beta$ for $a>0$. It is clear that $L_{a,b}$ lies in the locus where $s=a$ and $\rho=b^{2}$, and so $L_{a,b}\in U(R)$ if and only if $0<b^{2}/a^{2}\le R^{2}$.

Let $\Lambda$ denote the canonical form on $T^{*}(\R/\Z\times B)$. By definition, $L_{a,b}^{*}\Lambda=a\d t+b\beta$. Consider the symplectomorphism $f_{a,b}:T^{*}(\R/\Z\times B)\to T^{*}(\R/\Z\times B)$ which acts on one-forms by subtracting $(a-kb)\d t+b\beta$. It is clear that $f_{a,b}^{*}\Lambda=\Lambda-(a-kb)\d t-b\beta$, and hence:
\begin{equation*}
  [(f_{a,b}\circ \varphi\circ L_{a,b})^{*}\Lambda]=[L_{a,b}^{*}\Lambda]-a[\d t]-b[\beta]=0.
\end{equation*}
Consequently, $f_{a,b}\circ \varphi\circ L_{a,b}$ is a compact exact Lagrangian, and \emph{hence intersects the zero section}. In particular, $\varphi\circ L_{a,b}$ intersects the graph $f_{a,b}^{-1}(0_{B})$, i.e., $(a-kb)\d t+b\beta$. Thus $\rho/s^{2}$ attains the value $b^{2}/(a-kb)^{2}$ at some point on $\varphi(L_{a,b})$.

On the other hand, $\varphi(L_{a,b})$ is contained in $U(r)$. Since $L_{a,b}\in U(R)$ one has:
\begin{equation*}
  a>0\text{ and }b^{2}/a^{2}\le R^{2}\implies b^{2}/(a-kb)^{2}\le r^{2}
\end{equation*}
For $k\ge 0$, pick $a=1$ and $b=\min\set{1/k,R}$ to conclude a contradiction. For $k<0$, pick $b=\max\set{1/k,-R}$, and similarly conclude a contradiction. This completes the proof of Theorem \ref{theorem:nonsqueezing_gen}.

\bibliography{citations}
\bibliographystyle{alpha}
\end{document}